\documentclass[leqno]{aadmbook}

\usepackage{color}

\usepackage{amsthm}
\usepackage{amsfonts}
\usepackage{amsmath}
\newcommand{\dpq}{d_{p,q}}
\newcommand{\Dpq}{D_{p,q}}
\newcommand{\Epq}{\textsl{E}_{p,q}}
\newcommand{\epq}{\textsl{e}_{p,q}}

\newcommand{\pqbinomial}[4]{\mbox{$
\biggl[ \!\!
\begin{array}{c}
#1\\
 #2
\end{array}\!\!\biggr]_{
\!{#3,#4}} $} }

\newcommand{\N}{\ensuremath{\mathbb{N}}}

\newcommand{\Z}{\ensuremath{\mathbb{Z}}}

\newcommand{\C}{\ensuremath{\mathbb{C}}}

\newcommand{\partialpq}[2]{\dfrac{\partial_{p,q}^{#2}}{\partial_{p,q}{#1}^{#2}}}
\newcommand{\partialpqone}[1]{\dfrac{\partial_{p,q}}{\partial_{p,q}{#1}}}

\newtheorem{theorem}{Theorem}[section]

\newtheorem{proposition}{Proposition}[section]
\theoremstyle{definition}
\newtheorem{definition}[theorem]{Definition}

\newtheorem{corollary}[theorem]{Corollary}
\theoremstyle{remark}
\newtheorem{remark}[theorem]{Remark}

\numberwithin{equation}{section}

\begin{document}


\newcommand{\qhypergeom}[5]{\mbox{$
_#1 \phi_#2\left. \!\! \left( \!\!\!\!
\begin{array}{c}
\multicolumn{1}{c}{\begin{array}{c} #3
\end{array}}\\[1mm]
\multicolumn{1}{c}{\begin{array}{c} #4
            \end{array}}\end{array}
\!\!\!\! \right| \displaystyle{#5}\right) $} }

\newcommand{\pqhypergeom}[5]{\mbox{$
_#1 \Phi_#2\left. \!\! \left( \!\!\!\!\!\!\!
\begin{array}{c}
\multicolumn{1}{c}{\begin{array}{c} #3
\end{array}}\\[1mm]
\multicolumn{1}{c}{\begin{array}{c} #4
            \end{array}}\end{array}
\!\! \right| \displaystyle{#5}\!\right) $} }

\newcommand{\Cos}{\textrm{\ensuremath{{Cos}}}}
\newcommand{\Sin}{\textrm{\ensuremath{{Sin}}}}

\newcommand{\pqcos}{\cos_{p,q}}
\newcommand{\pqsin}{\sin_{p,q}}
\newcommand{\pqCos}{\Cos_{p,q}}
\newcommand{\pqSin}{\Sin_{p,q}}
\newcommand{\Lpq}{L_{p,q}}
\newcommand{\lpq}{\mathcal{L}_{p,q}}

\oddsidemargin 16.5mm
\evensidemargin 16.5mm

\thispagestyle{plain}

%
%

\vspace{5cc}
\begin{center}

{\large\bf ON TWO $(p,q)$-ANALOGUES OF THE LAPLACE TRANSFORM
\rule{0mm}{6mm}\renewcommand{\thefootnote}{}
\footnotetext{\scriptsize 2010 Mathematics Subject Classification. FILL SUBJECT MSCs HERE.

\rule{2.4mm}{0mm}Keywords and Phrases. $(p,q)$-Exponential, $(p,q)$-Laplace, $(p,q)$-integral, $(p,q)$-derivative.
}}

\vspace{1cc}
{\large\it P. Njionou Sadjang}

\vspace{1cc}
\parbox{24cc}{{\small

Two $(p,q)$-Laplace transforms are introduced and their relative properties are stated and proved. Applications are made to solve some $(p,q)$-linear difference equations. 

%

}}
\end{center}

\vspace{1cc}

\vspace{1cc}

\section{Introduction}

\noindent The classical Laplace transform of a function $f(t)$ is given by 
\begin{equation}\label{laplace1}
\mathcal{L}\{f(t)\}(s)=\int_{0}^{\infty}e^{-st}f(t)dt,\quad \quad s=a+ib\in\C,\end{equation}
and plays a fundamental role in pure and applied analysis, specially in solving differential equations. If a function of discrete variable $f(t)$, $t\in \Z$ is considered, then the integral transform \eqref{laplace1} reads 
\begin{equation}\label{ztrans}
  F(z)=\sum_{j=0}^{\infty}f(j)z^{-j},\quad z=e^{-p}.
\end{equation}
Equation \eqref{ztrans} is referred to as $Z$-transform and plays similiar role in difference analysis as Laplace transform in continuous analysis, specially in solving difference equations. 

In order to deal with $q$-difference equations, $q$-versions of the classical Laplace transform have been consecutively introduced in the literature. Studies of $q$-versions of Laplace transform go back to Hahn \cite{hahn}. Abdi \cite{abdi1960, abdi1962, abdi1964} published also many results in this domain. 

The $q$-deformed algebras \cite{Ogievetsky, quesne} and their generalizations ($(p,q)$-deformed algebras) \cite{burban, FLV, Jagannathan1998}  attract much attention these last years. The main reason is that these
topics stand for a meeting point of today's fast developing areas in mathematics and physics like the theory of quantum orthogonal polynomials and special functions, quantum groups, conformal field theories and statistics. From these works, many generalizations of special functions arise. There is a considerable list of references.

In this paper, we introduce two $(p,q)$-versions of the Laplace transform and provide some of their main properties. 	Next, some applications are done to solve some $(p,q)$-difference equations.

The paper is organised as follows:
\begin{enumerate}
   \item In Section 2, we recall basic notations, definitions and prove some impo4r4tant properties that will help in the next sections. The $(p,q)$-number, the $(p,q)$-factorial, the $(p,q)$-power, the $(p,q)$-binomial, the $(p,q)$-derivative, the $(p,q)$-integral, the $(p,q)$-exponentials, the $(p,q)$-trigonometric functions are successively introduced and some of their important properties are provided. 
   \item In Section 3, we introduce the $(p,q)$-Laplace transforms of first and of second kind. Their main properties are studied and the transforms of many fundemental functions are computed. 
   \item In Section 4, some applications of the Laplace transform of first kind are made. The same method can be used with the Laplace transform of second kind. The $(p,q)$-oscillator is introduced and solved using the Laplace transform of first kind. 
   \item In Section 5, we give a conclusion and indicate further possible directions that could be investigated to complete the following work. 
   \end{enumerate}

\section{Basic definitions and miscellaneous results  }

\subsection{$(p,q)$-number, $(p,q)$-factorial, $(p,q)$-binomial, $(p,q)$-power}
Let us introduce the following notation (see  \cite{JS2006},\cite{  JR2010},\cite{ njionou-2014-3})
\begin{equation*}\label{pqnumber}
[n]_{p,q}=\frac{p^n-q^n}{p-q},
\end{equation*}
for any positive integer. 

\noindent The twin-basic number is a natural generalization of the $q$-number, that is
\begin{equation*}
\lim\limits_{p\to 1}[n]_{p,q}=[n]_q.
\end{equation*}

\noindent The $(p,q)$-factorial is defined by  (\cite{JR2010, njionou-2014-3})
\begin{equation*}
[n]_{p,q}!=\prod_{k=1}^{n}[k]_{p,q}!,\quad n\geq 1,\quad [0]_{p,q}!=1.
\end{equation*}
\noindent Let us introduce also the so-called $(p,q)$-binomial coefficients
\begin{equation*}\label{pqbin}
\pqbinomial{n}{k}{p}{q}=\dfrac{[n]_{p,q}!}{[k]_{p,q}![n-k]_{p,q}!}, \quad 0\leq k\leq n.
\end{equation*}
 Note that as $p\to 1$, the $(p,q)$-binomial coefficients reduce to the $q$-binomial coefficients.\\
 It is clear by definition that 
 \begin{equation*}\label{binomial1}
 \pqbinomial{n}{k}{p}{q}=\pqbinomial{n}{n-k}{p}{q}.
 \end{equation*}
 
\noindent  Let us introduce also the so-called the $(p,q)$-powers \cite{njionou-2014-3}
\begin{eqnarray*}
 (x\ominus a)_{p,q}^n&=&(x-a)(px-aq)\cdots (xp^{n-1}-aq^{n-1}),\\
 (x\oplus a)_{p,q}^n&=&(x+a)(px+aq)\cdots (xp^{n-1}+aq^{n-1}).
\end{eqnarray*}

\noindent These definitions are extended to 
\begin{eqnarray*}\label{infinitpq}
(a\ominus b)_{p,q}^{\infty}=\prod_{k=0}^{\infty}(ap^k-q^kb)\\
(a\oplus b)_{p,q}^{\infty}=\prod_{k=0}^{\infty}(ap^k+q^kb)
\end{eqnarray*}
where the convergence is required.

\subsection{The $(p,q)$-derivative and the $(p,q)$-integral}
\begin{definition}\cite{njionou-2014-3}
Let $f$ be an arbitrary function and $a$ be a real number, then the $(p,q)$-integral of  $f$ is defined by
\begin{eqnarray}
\int_{0}^{a}f(x)\dpq x&=&(p-q)a\sum\limits_{k=0}^{\infty}\frac{q^{k}}{p^{k+1}}f\left(\frac{q^{k}}{p^{k+1}}a\right)  \quad \textrm{if }\quad \left|\dfrac{p}{q}\right|>1.\label{pqint2}
\end{eqnarray}
\end{definition}


\begin{definition}\cite{njionou-2014-3}
The improper $(p,q)$-integral of $f(x)$ on $[0;\infty)$ is defined to be 
\begin{eqnarray}
\int_0^{\infty}f(x)\dpq x
&=& (p-q)\sum_{j=-\infty}^{\infty}\dfrac{q^j}{p^{j+1}}f\left(\frac{q^{j}}{p^{j+1}}\right),\;\;\;  0<\dfrac{q}{p}<1.\label{improper1}
\end{eqnarray} 
\end{definition}

\noindent Let $f$ be a function defined on the set of the complex numbers.

\begin{definition}
The $(p,q)$-derivative of  the function $f$  is defined as  
\begin{equation*}
\Dpq f(x)=\dfrac{f(px)-f(qx)}{(p-q)x},\quad x\neq0,
\end{equation*}
and $(\Dpq f)(0)=f'(0)$,
provided that $f$ is differentiable at $0$.
\end{definition}

\begin{proposition}
The $(p,q)$-derivative fulfils the following product and quotient rules
\begin{eqnarray*}
\Dpq (f(x)g(x))&=& f(px)\Dpq g(x)+g(qx)\Dpq f(x),\label{productrule2}\\
\Dpq (f(x)g(x))&=& g(px)\Dpq f(x)+f(qx)\Dpq g(x).\label{productrule3}
\end{eqnarray*}
\begin{eqnarray*}
\Dpq \left(\frac{f(x)}{g(x)}\right)&=&\dfrac{g(qx)\Dpq f(x)-f(qx)\Dpq g(x)}{g(px)g(qx)},\label{quotient1}\\
\Dpq \left(\frac{f(x)}{g(x)}\right)&=&\dfrac{g(px)\Dpq f(x)-f(px)\Dpq g(x)}{g(px)g(qx)}.\label{quotient2}
\end{eqnarray*}
\end{proposition}

\begin{proposition}
Let $n$ be an integer $n\geq 0$, then the following formula applies
\begin{equation}\label{nth-deriv-invfunct}
\Dpq^n\left[\dfrac{1}{x}\right]=(-1)^n\dfrac{[n]_{p,q}!}{(pq)^{\binom{n+1}{2}}x^{n+1}}. 
\end{equation}
\end{proposition}

\begin{proof}
The relation is obvious for $n=0$. Let $n\geq 1$, assume that \eqref{nth-deriv-invfunct} holds true. Then 
\begin{eqnarray*}
\Dpq^{n+1}\left[\dfrac{1}{x}\right]&=&\Dpq\left[(-1)^n\dfrac{[n]_{p,q}!}{(pq)^{\binom{n+1}{2}}x^{n+1}}\right]\\
&=& \dfrac{(-1)^n[n]_{p,q}!}{(pq)^{\binom{n+1}{2}}}\times \dfrac{1}{(p-q)x}\left(\dfrac{1}{(px)^{n+1}}-\dfrac{1}{(qx)^{n+1}}\right)\\
&=& \dfrac{(-1)^n[n]_{p,q}!}{(pq)^{\binom{n+1}{2}}}\times\dfrac{-[n+1]_{p,q}}{(pq)^{n+1}x^{n+2}}=[(-1)^{n+1}\dfrac{[n+1]_{p,q}!}{(pq)^{\binom{n+2}{2}}x^{n+2}}.
\end{eqnarray*}
The proof is then complete.
\end{proof}
\noindent The next proposition generalizes \eqref{nth-deriv-invfunct}.

\begin{proposition}
$a$ is a non zero complex number. Then 
\begin{eqnarray}
\Dpq^n\left[\dfrac{1}{ax+b}\right]&=&\dfrac{(-a)^n[n]_{p,q}!}{\prod\limits_{k=0}^n\left(ap^{n-k}q^kx+b\right)}\label{gen-nth-deriv-invfunct}\\
&=&\dfrac{(-a)^n[n]_{p,q}!}{(ap^nx+b)(ap^{n-1}qx+b)\cdots(apq^{n-1}x+b)(aq^nx+b)}. \nonumber
\end{eqnarray}
\end{proposition}

\begin{proof}
The proof follows easily by induction. 
\end{proof}
\noindent Note that for $a=1$ and $b=0$, \eqref{gen-nth-deriv-invfunct} reduces to \eqref{nth-deriv-invfunct}.

\begin{proposition}\cite{njionou-2014-3} If $F(x)$ is a $(p,q)$-antiderivative of $f(x)$ and $F(x)$ is continuous at $x=0$, we have 
\begin{equation*}\label{fundemental}
\int_a^bf(x)\dpq x=F(b)-F(a),\quad 0\leq a<b\leq \infty.
\end{equation*}
\end{proposition}

\begin{corollary}\cite{njionou-2014-3}\label{cor1}
If $f'(x)$ exists in a neighbourhood  of $x=0$ and is continuous at $x=0$, where $f'(x)$ denotes the ordinary derivative of $f(x)$, we have 
\begin{equation*}\label{fundemantal2}
\int_{a}^{b}\Dpq f(x)\dpq x=f(b)-f(a).
\end{equation*}
\end{corollary}

\begin{proposition}\cite{njionou-2014-3} Suppose that $f(x)$ and $g(x)$ are two functions whose ordinary derivatives exist in a neighbourhood of $x=0$. $a$ and $b$ are two real numbers such that $a<b$, then
\begin{equation}\label{pq-int-part}
\int_a^b f(px)\left(\Dpq g(x)\right)\dpq x=f(b)g(b)-f(a)g(a)-\int_a^b g(qx)\left(\Dpq f(x)\right)\dpq x.
\end{equation}
\end{proposition}

\subsection{The $(p,q)$-hypergeometric functions}

Here, we give a natural generalization of the $q$-hypergeometric series (\cite{Gasper-Rahman})
{\small\begin{eqnarray*}
 \qhypergeom{r}{s}{ a_{1}, a_{2},\cdots,a_{r}} {b_{1}, b_{2},\cdots ,
b_{s}}{q;z} &=& \sum\limits_{n=0}^{\infty
}\frac{(a_1;q)_n\cdots(a_r;q)_n}{(q,q)_n(b_1;q)_n\cdots(b_s;q)_n}\Big[(-1)^nq^{\binom
n2}\Big]^{1+s-r}z^n.
\end{eqnarray*}}

\begin{definition}
The $(p,q)$-hypergeometric series 
{\small \begin{eqnarray}\label{pqseries}
&&\pqhypergeom{r}{s}{(a_{1p},a_{1q}),\ldots (a_{rp},a_{rq})}{(b_{1p},b_{1q}),\ldots,(b_{sp},b_{sq})}{p,q;z} \\
&&\nonumber\hspace{1.5cm}=\sum_{n=0}^{\infty}\frac{(a_{1p}\ominus a_{1q})_{p,q}^n\cdots (a_{rp}\ominus a_{rq})_{p,q}^n}{(b_{1p}\ominus b_{1q})_{p,q}^n\cdots (b_{sp}\ominus b_{sq})_{p,q}^n(p\ominus q)_{p,q}^n}\Big[(-1)^n\left(\frac{q}{p}\right)^{\binom
n2}\Big]^{1+s-r}z^n,
\end{eqnarray}}
(\emph{compare \cite{Jagannathan1998,JS2006}}).
\end{definition}

\begin{theorem}[Compare to \cite{JS2006}]
Let $a$, $b$ be two  complex numbers, then we have the following 
\begin{equation}\label{pqbinomtheo}
\pqhypergeom{1}{0}{(a,b)}{-}{p,q;z}=\sum_{n=0}^{\infty}\dfrac{(a\ominus b)_{p,q}^{n}}{(p\ominus q)_{p,q}^n}z^n=\dfrac{(p\ominus bz)_{p,q}^{\infty}}{(p\ominus az)_{p,q}^{\infty}}. 
\end{equation}
\end{theorem}

\begin{proof} We first note that $\dfrac{(a\ominus b)_{p,q}^n}{(p\ominus q)_{p,q}^n}=\dfrac{\left(\frac{b}{a};\frac{q}{p}\right)_n}{\left(\frac{q}{p};\frac{q}{p}\right)_n}\left( \dfrac{a}{p} \right)^{n}$. It follows from the $q$-binomial theorem (see \cite{Gasper-Rahman}) that  
\begin{eqnarray*}
\sum_{n=0}^{\infty}\dfrac{(a\ominus b)_{p,q}^{n}}{(p\ominus q)_{p,q}^n}z^n&=& \sum_{n=0}^{\infty}\dfrac{\left(\frac{b}{a};\frac{q}{p}\right)_n}{\left(\frac{q}{p};\frac{q}{p}\right)_n}\left( \dfrac{az}{p} \right)^{n}\\
&=& \dfrac{\left(\frac{bz}{p};\frac{q}{p}\right)_\infty}{\left(\frac{az}{p};\frac{q}{p}\right)_\infty}=\dfrac{(p\ominus bz)_{p,q}^{\infty}}{(p\ominus az)_{p,q}^{\infty}}.
\end{eqnarray*}
\end{proof}

\noindent The following corollary also appears in \cite{JS2006}.

\begin{corollary}
$a$, $b$ and $c$ are three complex numbers. Then 
\[\pqhypergeom{1}{0}{(a,b)}{-}{p,q;z}\pqhypergeom{1}{0}{(b,c)}{-}{p,q;z}=\pqhypergeom{1}{0}{(a,c)}{-}{p,q;z}.\]
\end{corollary}

\subsection{$(p,q)$-exponential and $(p,q)$-trigonometric functions}

\noindent As in the $q$-case, there are many definitions of the $(p,q)$-exponential function. The following two $(p,q)$-analogues of the exponential function (see \cite{JS2006}) will be frequently used throughout this paper:
\begin{eqnarray}
 e_{p,q}(z)&=&\pqhypergeom{1}{0}{(1,0)}{-}{p,q;(p-q)z}=\sum_{n=0}^{\infty}\dfrac{p^{\binom{n}{2}}}{[n]_{p,q}!}z^n,\label{pqexp}\\
 E_{p,q}(z)&=&\pqhypergeom{1}{0}{(0,1)}{-}{p,q;(q-p)z}=\sum_{n=0}^{\infty}\dfrac{q^{\binom{n}{2}}}{[n]_{p,q}!}z^n.\label{bigpqexp}
\end{eqnarray}
From the $(p,q)$-binomial theorem \eqref{pqbinomtheo} and the definitions \eqref{pqexp} and \eqref{bigpqexp} of the $(p,q)$-exponential functions, it is easy to see that 

\begin{equation}\label{invpqexpo}
e_{p,q}(x)E_{p,q}(-x)=1. 
\end{equation}

\noindent The next two propositions give the $n$-th derivative of the $(p,q)$-exponential functions.  These formulas are very important for the computations the $(p,q)$-Laplace transforms of some functions in the next sections. 
\begin{proposition} Let $\lambda$ be a complex number, then
the following relations hold
\begin{eqnarray*}
 \Dpq e_{p,q}(\lambda x)=\lambda e_{p,q}(\lambda px),\\
 \Dpq E_{p,q}(\lambda x)=\lambda E_{p,q}(\lambda qx).
\end{eqnarray*}
\end{proposition}

\begin{proof}
The proof follows from the definitions of the $(p,q)$-exponentials and the $(p,q)$-derivative.
\end{proof}

\begin{proposition}
Let $n$ be a nonnegative integer, then  the following equations hold 
\begin{eqnarray}
 \Dpq^n e_{p,q}(\lambda x)=\lambda^n p^{\binom{n}{2}} e_{p,q}(\lambda p^nx),\label{nth-deriv-pqExp}\\ 
 \Dpq^n E_{p,q}(\lambda x)=\lambda^n q^{\binom{n}{2}}\lambda E_{p,q}(\lambda q^nx). \label{nth-deriv-pqexp}
\end{eqnarray}
\end{proposition}

\begin{proof}
The proof follows by induction from the definitions of the $(p,q)$-exponentials and the $(p,q)$-derivative.
\end{proof}

\noindent From \eqref{pqexp} we can derive 
\begin{equation}\label{pq-cs}
e_{p,q}(iz)=\sum_{n=0}^{\infty}\dfrac{p^{\binom{n}{2}}}{[n]_{p,q}!}(iz)^n=\sum_{n=0}^{\infty}\dfrac{(-1)^np^{\binom{2n}{2}}}{[2n]_{p,q}!}z^{2n}+i\sum_{n=0}^{\infty}\dfrac{(-1)^np^{\binom{2n+1}{2}}}{[2n+1]_{p,q}!}z^{2n+1}.
\end{equation}

\noindent By \eqref{pq-cs}, we define the $(p,q)$-cosine and the $(p,q)$-sine functions as follows:
\begin{eqnarray}
   \cos_{p,q}(z)&=&\sum_{n=0}^{\infty}\dfrac{(-1)^np^{\binom{2n}{2}}}{[2n]_{p,q}!}z^{2n}, \label{pqcos1}\\
   \sin_{p,q}(z)&=&\sum_{n=0}^{\infty}\dfrac{(-1)^np^{\binom{2n+1}{2}}}{[2n+1]_{p,q}!}z^{2n+1}.\label{pqsin1}
\end{eqnarray}

\noindent Analogously,  from \eqref{bigpqexp} we can derive
\begin{equation}\label{pq-cs1}
E_{p,q}(iz)=\sum_{n=0}^{\infty}\dfrac{q^{\binom{n}{2}}}{[n]_{p,q}!}(iz)^n=\sum_{n=0}^{\infty}\dfrac{(-1)^nq^{\binom{2n}{2}}}{[2n]_{p,q}!}z^{2n}+i\sum_{n=0}^{\infty}\dfrac{(-1)^nq^{\binom{2n+1}{2}}}{[2n+1]_{p,q}!}z^{2n+1}.
\end{equation} 

\noindent And by \eqref{pq-cs}, we define the big $(p,q)$-cosine and the big $(p,q)$-sine functions  as follows:
\begin{eqnarray}
   \Cos_{p,q}(z)&=&\sum_{n=0}^{\infty}\dfrac{(-1)^nq^{\binom{2n}{2}}}{[2n]_{p,q}!}z^{2n},\label{pqcosBig}\\
   \Sin_{p,q}(z)&=&\sum_{n=0}^{\infty}\dfrac{(-1)^nq^{\binom{2n+1}{2}}}{[2n+1]_{p,q}!}z^{2n+1}.\label{pqsinBig}
\end{eqnarray}

\noindent It is easy to see that 
\[\cos_{p,q}(z)=\Cos_{q,p}(z)\quad \textrm{and}\quad \sin_{p,q}(z)=\Sin_{q,p}(z).\]
Clearly, 
\begin{eqnarray*}
\Dpq\cos_{p,q}(z)&=&-\sin_{p,q}(qz),\\
\Dpq\sin_{p,q}(z)&=&\cos_{p,q}(pz),\\
\Dpq\Cos_{p,q}(z)&=&-\Sin_{p,q}(qz),\\
\Dpq\Sin_{p,q}(z)&=&\Cos_{p,q}(qz).
\end{eqnarray*}

\begin{proposition}
The following equations hold
\begin{eqnarray*}
&&\pqcos(x)\pqCos(x)+\pqsin(x)\pqSin(x)=1,\\
&&\pqsin(x)\pqCos(x)-\pqcos(x)\pqSin(x)=0.
\end{eqnarray*}
\end{proposition}

\begin{proof}
The proof follows from \eqref{invpqexpo}.
\end{proof}

\noindent Let us now define the hyperbolic $(p,q)$-cosine and the hyperbolic $(p,q)$-sine functions as follows 
\begin{eqnarray}
   \cosh_{p,q}(z)&=&\dfrac{e_{p,q}(z)+e_{p,q}(-z)}{2}=\sum_{n=0}^{\infty}\dfrac{p^{\binom{2n}{2}}}{[2n]_{p,q}!}z^{2n},\label{hyper-cos}\\
   \sinh_{p,q}(z)&=&\dfrac{e_{p,q}(z)-e_{p,q}(-z)}{2}=\sum_{n=0}^{\infty}\dfrac{p^{\binom{2n+1}{2}}}{[2n+1]_{p,q}!}z^{2n+1},\label{hyper-sin}\\
   Cosh_{p,q}(z)&=&\dfrac{E_{p,q}(z)+E_{p,q}(-z)}{2}=\sum_{n=0}^{\infty}\dfrac{q^{\binom{2n}{2}}}{[2n]_{p,q}!}z^{2n},\label{hyper-big-cos}\\
   Sinh_{p,q}(z)&=&\dfrac{E_{p,q}(z)-E_{p,q}(-z)}{2}=\sum_{n=0}^{\infty}\dfrac{q^{\binom{2n+1}{2}}}{[2n+1]_{p,q}!}z^{2n+1}.\label{hyper-big-sin}
\end{eqnarray}

\begin{proposition}
The following equations hold
\begin{eqnarray*}
&& \cosh_{p,q}(z)Cosh_{p,q}(z)-\sinh_{p,q}(z)Sinh_{p,q}(z)=1,\\
&& \cosh_{p,q}(z)Sinh_{p,q}(z)-\sinh_{p,q}(z)Cosh_{p,q}(z)=0.
\end{eqnarray*}
\end{proposition}

\begin{proof}
The proof follows from \eqref{invpqexpo}.
\end{proof}

\subsection{$(p,q)$-Gamma function}

\begin{definition}\label{pqgamma}\cite{njionou-2015-1}
Let $x$ be a complex number, we define the  $(p,q)$-Gamma function as
\begin{equation}\label{pqgam}
\Gamma_{p,q}(x)=\dfrac{(p\ominus q)_{p,q}^{\infty}}{(p^x\ominus q^x)^\infty_{p,q}}(p-q)^{1-x},\; 0<q<p.
\end{equation}
\end{definition}

\begin{proposition}\cite{njionou-2015-1}
The $(p,q)$-Gamma function fulfils the following fundemental relation
\begin{equation}\label{gammarec}
\Gamma_{p,q}(x+1)=[x]_{p,q}\Gamma_{p,q}(x).
\end{equation}

\end{proposition}

\begin{remark}
If $n$ is a nonnegative integer, it follows from (\ref{gammarec}) that 
\[\Gamma_{p,q}(n+1)=[n]_{p,q}!.\]
It can be also easyly seen from the definition that
\[\Gamma_{p,q}(n+1)=\dfrac{(p\ominus q)^n_{p,q}}{(p-q)^n}.\]
\end{remark}

\noindent Very recently, a $(p,q)$-integral representation of the $(p,q)$-Gamma function was given in \cite{aral} when the argument is a nonnegative integer as follows
\begin{equation}\label{pq-gamma-integral}
\Gamma_{p,q}(n)=\int_{0}^{\infty}p^{\frac{(n-1)(n-2)}{2}}t^{n-1}E_{p,q}(-qt)\dpq t. 
\end{equation}

\noindent \textcolor{black}{ Note that in this definition, there is maybe a mistake, the factor $p^{\frac{(n-1)(n-2)}{2}}$ should be replaced by $p^{\frac{n(n-1)}{2}}$ so we can clearly get $\Gamma_{p,q}(n+1)=[n]_{p,q}\Gamma_{p,q}(n)$. } Relation \eqref{pq-gamma-integral} enables to prove \eqref{gammarec} again using the formula of the $(p,q)$-integration by part \eqref{pq-int-part}.

\noindent Now, we propose another definition of the $(p,q)$-Gamma function which will be frequently use throughout the text.  

\begin{definition}
For $0<q<p$, we define a $(p,q)$-Gamma function by
\textcolor{black}{
\begin{equation}\label{pq-gam-int2}
\Gamma_{p,q}(z)=p^{\frac{z(z-1)}{2}}\int_{0}^{\infty}t^{z-1}E_{p,q}(-qt)\dpq t.
\end{equation}}
\end{definition} 

\begin{proposition}
Let $z$ be a complex number such that $\Gamma_{p,q}(z+1)$ and $\Gamma_{p,q}(z)$ exist, then
\begin{equation}
\Gamma_{p,q}(z+1)=[z]_{p,q}\Gamma_{p,q}(z).
\end{equation}
\end{proposition}

\begin{proof}
Using equation \eqref{pq-gam-int2} and the $(p,q)$-integration by part \eqref{pq-int-part}, we have:
\textcolor{black}{
\begin{eqnarray*}
\Gamma_{p,q}(z+1)&=& p^{\frac{z(z+1)}{2}}\int_{0}^{\infty}t^{z}E_{p,q}(-qt)\dpq t\\
&=& -p^{\frac{z(z-1)}{2}}\int_{0}^{\infty}(pt)^{z}\Dpq E_{p,q}(-t)\dpq t\\
&=& -p^{\frac{z(z-1)}{2}}\left[ t^{z}E_{p,q}(-t)\right]_{0}^{\infty}+p^{\frac{z(z-1)}{2}}[z]_{p,q}\int_{0}^{\infty}t^{z-1}E_{p,q}(-qt)\dpq t\\
&=&[z]_{p,q}\Gamma_{p,q}(z).
\end{eqnarray*}}
\end{proof}

\section{Two $(p,q)$-Laplace transforms}

\subsection{The $(p,q)$-Laplace transform of the first kind}

\begin{definition}
For a given function $f(t)$, we define its $(p,q)$-Laplace transform of the first kind as the function 
\begin{equation}\label{pqlaplace1}
F(s)=L_{p,q}\{f(t)\}(s)=\int_{0}^{\infty}f(t)E_{p,q}(-qts)\dpq t, \;\; s>0.
\end{equation}
\end{definition}

\begin{proposition}
For any two complex numbers $\alpha$ and $\beta$, we have 
\begin{equation*}
L_{p,q}\{\alpha f(t)+\beta g(t)\}=\alpha L_{p,q}\{f(t)\}+\beta L_{p,q}\{g(t)\}.
\end{equation*}
\end{proposition}

\begin{proof}
The proof follows by \eqref{pqlaplace1}. 
\end{proof}

\noindent In what follows, we give some examples. From \eqref{pqlaplace1}, we note that:
\begin{eqnarray*}
L_{p,q}\{1\}(s)&=&\int_{0}^{\infty}E_{p,q}(-qst)\dpq t=-\dfrac{1}{s}\int_{0}^{\infty}\Dpq E_{p,q}(-st)\dpq t\\
&=& -\dfrac{1}{s}\left[ E_{p,q}(-st)\right]_0^{\infty}=\dfrac{1}{s},\quad s>0.\\
L_{p,q}\{t\}(s)&=&\int_{0}^{\infty}t E_{p,q}(-qst)\dpq t=-\dfrac{1}{ps}\int_{0}^{\infty}(pt)\Dpq E_{p,q}(-st)\dpq t\\
&=&-\dfrac{1}{ps}\left\{\left[tE_{p,q}(-st)\right]_{0}^{\infty}-\int_{0}^{\infty}E_{p,q}(-qst)\dpq t\right\}\\
&=&\dfrac{1}{ps^2},\quad s>0.\\
L_{p,q}\{1+5t\}(s)&=&L_{p,q}\{1\}(s)+5L_{p,q}\{t\}(s)=\dfrac{1}{s}+\dfrac{5}{ps^2},\quad s>0.
\end{eqnarray*}

\begin{proposition}\label{dil1}
Let $\alpha$ be a non zero complex number, then 
\begin{eqnarray}
\int_{0}^{\infty}f(\alpha t)\dpq t=\dfrac{1}{\alpha}\int_{0}^{\infty}f(t)\dpq t.\label{changevar}
\end{eqnarray} 
\end{proposition}

\begin{theorem}[Scaling]\label{theo-dil}
Let $a$ be a non zero complex number, then the following formula applies 
\begin{equation}\label{theo-dil-equation}
L_{p,q}\{f(at)\}(s)=\dfrac{1}{a}L_{p,q}\{f(t)\}\left(\frac{s}{a}\right).
\end{equation}
\end{theorem}

\begin{proof}
Using the definition and Proposition \ref{dil1}, we have
\begin{eqnarray*}
L_{p,q}\{f(at)\}(s)&=&\int_{0}^{\infty}f(at)\Epq(-qst)\dpq t\\
&=& \int_{0}^{\infty}f(at)\Epq(-aq\frac{s}{a}t)\dpq t\\
&=&\frac{1}{a}\int_{0}^{\infty}f(t)\Epq(-q\frac{s}{a}t)\dpq t=\dfrac{1}{a}L_{p,q}\{f(t)\}\left(\frac{s}{a}\right).
\end{eqnarray*} 
\end{proof}

\begin{theorem}\label{Theo-laplace-power}
For $\alpha>-1$, we have the following 
\begin{equation}
L_{p,q}(t^\alpha)=\dfrac{\Gamma_{p,q}(\alpha+1)}{p^{\frac{\alpha(\alpha+1)}{2}}s^{\alpha+1}}.
\end{equation}
\end{theorem}

\begin{proof}
We have 
\begin{eqnarray*}
 L_{p,q}\{t^\alpha\}(s)&=& \int_{0}^{\infty}t^\alpha E_{p,q}(-qst)\dpq t=\dfrac{1}{s^{\alpha+1}}\int_{0}^{\infty}E_{p,q}(-qt)t^{\alpha}\dpq t\\
 &=&\dfrac{1}{p^{\frac{\alpha(\alpha+1)}{2}}s^{\alpha+1}}\int_{0}^{\infty}p^{\frac{\alpha(\alpha+1)}{2}}t^{(\alpha+1)-1}E_{p,q}(-qt)\dpq t\\
 &=&\dfrac{\Gamma_{p,q}(\alpha+1)}{p^{\frac{\alpha(\alpha+1)}{2}}s^{\alpha+1}}.
\end{eqnarray*}
\end{proof}

\noindent The following theorem is a particular case of Theorem \eqref{Theo-laplace-power} when $\alpha=n$ is a nonnegative integer. 
\begin{theorem}
Let $n\in\N$, then for $s>0$, we have 
\begin{equation}\label{trans01}
L_{p,q}\{t^n\}(s)=\dfrac{[n]_{p,q}!}{p^{\binom{n+1}{2}}s^{n+1}}.
\end{equation}
\end{theorem}

\begin{proof}
We provide a proof by induction for this result. The result is obvious for $n=0$. Assume that it holds true for some nonnegative integer $n$, then using the $(p,q)$-integration by part \eqref{pq-int-part}, we have
\begin{eqnarray*}
    L_{p,q}\{t^{n+1}\}(s)&=& \int_{0}^{\infty}t^{n+1}E_{p,q}(-qst)\dpq t\\
    &=&-\dfrac{1}{p^{n+1}s}\int_{0}^{\infty}(pt)^{n+1}\Dpq E_{p,q}(-ts)\dpq t\\
    &=& -\dfrac{1}{p^{n+1}s}\left\{ \left[t^{n+1}E_{p,q}(-st)\right]_0^\infty- [n+1]_{p,q}\int_{0}^{\infty}t^{n}E_{p,q}(-qts)\dpq t\right\}\\
    &=&\dfrac{[n+1]_{p,q}}{p^{n+1}s}L_{p,q}\{t^n\}(s)\\
    &=& \dfrac{[n+1]_{p,q}}{p^{n+1}s}\dfrac{[n]_{p,q}!}{p^{\binom{n+1}{2}}s^{n+1}}\\
    &=&\dfrac{[n+1]_{p,q}!}{p^{\binom{n+2}{2}}s^{n+2}}.
\end{eqnarray*}
This proves the assertion.
\end{proof}

\noindent Next, we give explicit formulas for the transform of the $(p,q)$-exponential and the $(p,q)$-trigonometric functions. 

\begin{theorem}
Let $a$ be a real number, then
\begin{eqnarray}
   L_{p,q}\{e_{p,q}(at)\}(s)&=&\dfrac{p}{ps-a}, \quad s>\dfrac{a}{p},\label{exp-trans}\\
   L_{p,q}\{E_{p,q}(at)\}(s)&=& \dfrac{1}{s}\sum_{n=0}^{\infty}(-1)^n\left(\dfrac{q}{p}\right)^{\binom{n}{2}}\left(\dfrac{a}{ps}\right)^{n}.
\end{eqnarray}
\end{theorem}

\begin{proof}
Using \eqref{pqexp}, \eqref{bigpqexp} and  \eqref{trans01}, we have
\begin{eqnarray*}
  L_{p,q}\{e_{p,q}(at)\}(s)&=& \int_{0}^{\infty}E_{p,q}(-qst)e_{p,q}(at)\dpq t\\
  &=& \sum_{n=0}^{\infty}\dfrac{a^np^{\binom{n}{2}}}{[n]_{p,q}!}\int_{0}^{\infty}E_{p,q}(-qst)t^n\dpq t\\
  &=&\sum_{n=0}^{\infty}\dfrac{a^np^{\binom{n}{2}}}{[n]_{p,q}!}\dfrac{[n]_{p,q}!}{p^{\binom{n+1}{2}}s^{n+1}}\\
  &=&\dfrac{1}{s}\sum_{n=0}^{\infty}\left(\dfrac{a}{ps}\right)^n=\dfrac{p}{ps-a}.
\end{eqnarray*}
\begin{eqnarray*}
 L_{p,q}\{E_{p,q}(at)\}(s)&=&\int_{0}^{\infty}E_{p,q}(-qst)E_{p,q}(at)\dpq t\\
  &=& \sum_{n=0}^{\infty}(-1)^n\dfrac{a^nq^{\binom{n}{2}}}{[n]_{p,q}!}\int_{0}^{\infty}E_{p,q}(-qst)t^n\dpq t\\
   &=&\sum_{n=0}^{\infty}(-1)^n\dfrac{a^nq^{\binom{n}{2}}}{[n]_{p,q}!}\dfrac{[n]_{p,q}!}{p^{\binom{n+1}{2}}s^{n+1}}\\
   &=&\dfrac{1}{s}\sum_{n=0}^{\infty}(-1)^n\left(\dfrac{q}{p}\right)^{\binom{n}{2}}\left(\dfrac{a}{ps}\right)^{n}.
\end{eqnarray*}
\end{proof}

\begin{theorem}
The following relations apply 
\begin{eqnarray*}
     L_{p,q}\{\cos_{p,q}(at)\}(s)&=& \dfrac{p^2s}{(ps)^2+a^2},\\
      L_{p,q}\{\sin_{p,q}(at)\}(s)&=& \dfrac{pa}{(ps)^2+a^2}.
\end{eqnarray*}
\end{theorem}

\begin{proof}
Using equations \eqref{pqcos1}, \eqref{pqsin1} and \eqref{pqlaplace1}, we have: 
\begin{eqnarray*}
  L_{p,q}\{\cos_{p,q}(at)\}(s)&=& \int_{0}^{\infty}E_{p,q}(-qst)\cos_{p,q}(at)\dpq t\\
  &=& \sum_{n=0}^{\infty}\dfrac{(-1)^na^{2n}p^{\binom{2n}{2}}}{[2n]_{p,q}!}\int_{0}^{\infty}E_{p,q}(-qst)t^{2n}\dpq t\\
  &=&\sum_{n=0}^{\infty}\dfrac{(-1)^na^{2n}p^{\binom{2n}{2}}}{[2n]_{p,q}!}\dfrac{[2n]_{p,q}!}{p^{\binom{2n+1}{2}}s^{2n+1}}\\
  &=&\dfrac{1}{s}\sum_{n=0}^{\infty}(-1)^n\left(\dfrac{a}{ps}\right)^{2n}=\dfrac{p^2s}{(ps)^2+a^2}.
\end{eqnarray*}
\begin{eqnarray*}
  L_{p,q}\{\sin_{p,q}(at)\}(s)&=& \int_{0}^{\infty}E_{p,q}(-qst)\sin_{p,q}(at)\dpq t\\
  &=& \sum_{n=0}^{\infty}\dfrac{(-1)^na^{2n+1}p^{\binom{2n+1}{2}}}{[2n+1]_{p,q}!}\int_{0}^{\infty}E_{p,q}(-qst)t^{2n+1}\dpq t\\
  &=&\sum_{n=0}^{\infty}\dfrac{(-1)^na^{2n+1}p^{\binom{2n+1}{2}}}{[2n+1]_{p,q}!}\dfrac{[2n+1]_{p,q}!}{p^{\binom{2n+2}{2}}s^{2n+2}}\\
  &=&\dfrac{1}{s}\sum_{n=0}^{\infty}(-1)^n\left(\dfrac{a}{ps}\right)^{2n+1}=\dfrac{pa}{(ps)^2+a^2}.
\end{eqnarray*}
\end{proof}

\begin{remark}
Note that one could also use \eqref{exp-trans}, \eqref{pqcos1} and \eqref{pqsin1} to obtain the result. 
\end{remark}

\begin{theorem}
The following equations apply
\begin{eqnarray*}
     L_{p,q}\{\cosh_{p,q}(at)\}(s)&=& \dfrac{p^2s}{(ps)^2-a^2},\quad  s>\left|\dfrac{a}{p}\right|\\
      L_{p,q}\{\sinh_{p,q}(at)\}(s)&=& \dfrac{pa}{(ps)^2-a^2},\quad s>\left|\dfrac{a}{p}\right|.
\end{eqnarray*}
\end{theorem}

\begin{proof}
Using \eqref{pqlaplace1}, \eqref{hyper-cos} and \eqref{hyper-sin} we have
\begin{eqnarray*}
 L_{p,q}\{\cosh_{p,q}(at)\}(s)&=&\dfrac{1}{2} \left\{ L_{p,q}\{e_{p,q}(at)\}(s)+L_{p,q}\{e_{p,q}(-at)\}(s)\right\}\\
 &=&\dfrac{1}{2}\left(\dfrac{p}{ps-a}+\dfrac{p}{ps+a}\right)\\
 &=& \dfrac{p^2s}{(ps)^2-a^2},
\end{eqnarray*}
\begin{eqnarray*}
 L_{p,q}\{\sinh_{p,q}(at)\}(s)&=&\dfrac{1}{2} \left\{ L_{p,q}\{e_{p,q}(at)\}(s)-L_{p,q}\{e_{p,q}(-at)\}(s)\right\}\\
 &=&\dfrac{1}{2}\left(\dfrac{p}{ps-a}-\dfrac{p}{ps+a}\right)\\
 &=& \dfrac{pa}{(ps)^2-a^2}.
\end{eqnarray*}
\end{proof}

\noindent Next, $f$ being a function, we provide some properties related to the $(p,q)$-derivative of the $(p,q)$-Laplace transform of $f$ and the $(p,q)$-Laplace transform of the $(p,q)$-derivative of $f$. Let us introduce the following notation which makes clear the relative variable on which the $(p,q)$-derivative is applied: 
\[\partialpqone{s}f(x,s)=\dfrac{f(x,ps)-f(x,qs)}{(p-q)s},\] and \[ \partialpq{s}{n+1}=\partialpq{s}{n}\circ\partialpqone{s}, \quad n\geq 1,\quad\textrm{and}\quad\ \partialpq{s}{0} f=f. \]

\begin{theorem}[$(p,q)$-derivative of transforms]\label{theo-der-trans-1} For $n\in\N$, we have 
\begin{equation}\label{trans-tft1}
L_{p,q}\{t^nf(t)\}(s)=(-1)^nq^{\binom{n}{2}}\ \partialpq{s}{n}\left[F\left(q^{-n}s\right)\right].
\end{equation}
\end{theorem}

\begin{proof}
The result is obvious for $n=0$. Let $n\geq 1$, we have  
\begin{eqnarray*}
\partialpq{s}{n}[F(q^{-n}s)]&=& \int_{0}^{\infty}\! \partialpq{s}{n}\left[ E_{p,q}(-q^{-n+1}st)\right]f(t)\dpq t
\end{eqnarray*}
Using equation \eqref{nth-deriv-pqexp}, it follows that 
\begin{eqnarray*}
    \partialpq{s}{n}\left[ E_{p,q}(-q^{-n+1}st)\right]&=&\prod_{j=0}^{n-1}\left(-q^{n-1-j}t\right)E_{p,q}(-qst)\\
    &=&(-1)^nq^{-\binom{n}{2}}t^nE_{p,q}(-qst).
\end{eqnarray*}
The proof is then completed.
\end{proof}

\noindent Note that \eqref{trans01} can be obtained using Theorem \ref{theo-der-trans-1}. Of course, taking $f(t)=1$ in \eqref{trans-tft1} and using \eqref{nth-deriv-invfunct}, we have $F(s)=\dfrac{1}{s}$ and
\[L_{p,q}\{t^n\}(s)=(-1)^nq^{\binom{n}{2}}\ \partialpq{s}{n}\left[\dfrac{q^n}{s}\right]=(-1)^nq^{\binom{n+1}{2}}\dfrac{(-1)^n[n]_{p,q}!}{(pq)^{\binom{n+1}{2}}s^{n+1}}=\dfrac{[n]_{p,q}!}{p^{\binom{n+1}{2}}s^{n+1}}.\]

\begin{corollary}
The following equation applies:
\begin{eqnarray*}
\Lpq\{t^n\epq(at)\}(s)&=&\dfrac{p^{n+1}q^{\binom{n+1}{2}}[n]_{p,q}!}{(p^{n+1}s-aq^n)(p^nqs-aq^n)\cdots(p^2q^{n-1}s-aq^n)(pq^ns-aq^n)}\\
&=&\dfrac{p^{n+1}q^{\binom{n+1}{2}}[n]_{p,q}!}{\prod\limits_{k=0}^{n}\left(p^{n+1-k}q^ks-aq^n\right)}.
\end{eqnarray*}
\end{corollary}

\begin{proof}
The proof follows from \eqref{gen-nth-deriv-invfunct} and \eqref{trans-tft1}. 
\end{proof}

\begin{theorem}[Transform of the $(p,q)$-derivative]
The following transform rule applies.
\begin{equation}\label{trans-pqder1}
\Lpq\left\{\Dpq^n f(t)\right\}(s)=\dfrac{s^n}{p^{\binom{n+1}{2}}}\Lpq\{f(t)\}\left(\frac{s}{p^n}\right)-\sum_{k=0}^{n-1}\dfrac{s^{n-1-k}}{p^{\binom{n-k}{2}}}(\Dpq^k f)(0). 
\end{equation}
\end{theorem}

\begin{proof}
Let $f$ be a functions for which the $(p,q)$-Laplace transform exists. Then, for $n=1$, 
\begin{eqnarray*}
L_{p,q}\left\{\Dpq f(t)\right\}(s)&=& \int_{0}^{\infty}\Epq(-qst)\Dpq f(t)\dpq t\\
&=& \left[ f(t)\Epq(-st)\right]_{0}^{\infty}-\int_{0}^{\infty} f(pt)\Dpq \Epq(-st)\dpq t\\
&=&-f(0)+s\int_{0}^{\infty} f(pt)\Epq(-qst)\dpq t\\
&=&-f(0)+\frac{s}{p}\int_{0}^{\infty} f(t)\Epq(-q\frac{s}{p}t)\dpq t\\
&=&-f(0)+\dfrac{s}{p}L_{p,q}\{f(t)\}\left(\frac{s}{p}\right).
\end{eqnarray*}
Let $n\geq 1$, assume \eqref{trans-pqder1} holds true. Then, applying the result for $n=1$ with $\Dpq^n f(t)$, we have
\begin{eqnarray*}
\Lpq\left\{\Dpq^{n+1} f(t)\right\}(s)&=&-(\Dpq^n f)(0)+\dfrac{s}{p}L_{p,q}\{\Dpq^nf(t)\}\left(\frac{s}{p}\right)\\
&=& -(\Dpq^n f)(0)+\dfrac{s}{p}\left\{\dfrac{s^n}{p^{\binom{n+1}{2}+n}}\Lpq\{f(t)\}\left(\frac{s}{p^{n+1}}\right)\right.\\ &&\left.-\sum_{k=0}^{n-1}\dfrac{s^{n-1-k}}{p^{\binom{n-k}{2}+n-1-k}}(\Dpq^k f)(0)\right\}\\
&=& -(\Dpq^n f)(0)+\left\{\dfrac{s^{n+1}}{p^{\binom{n+1}{2}+n+1}}\Lpq\{f(t)\}\left(\frac{s}{p^{n+1}}\right)\right.\\ &&\left.-\sum_{k=0}^{n-1}\dfrac{s^{n-k}}{p^{\binom{n-k}{2}+n-k}}(\Dpq^k f)(0)\right\}\\
&=& -(\Dpq^n f)(0)+\left\{\dfrac{s^{n+1}}{p^{\binom{n+2}{2}}}\Lpq\{f(t)\}\left(\frac{s}{p^{n+1}}\right)\right.\\ &&\left.-\sum_{k=0}^{n-1}\dfrac{s^{n-k}}{p^{\binom{n-k+1}{2}}}(\Dpq^k f)(0)\right\}\\
&=&\dfrac{s^{n+1}}{p^{\binom{n+2}{2}}}\Lpq\{f(t)\}\left(\frac{s}{p^{n+1}}\right)-\sum_{k=0}^{n}\dfrac{s^{n-k}}{p^{\binom{n-k+1}{2}}}(\Dpq^k f)(0)
\end{eqnarray*}
This completes the proof.
\end{proof}

As a direct application, observe that taking $f(t)=t^n$ in \eqref{trans-pqder1}, we have 
\[\Lpq\{\Dpq^nt^n\}(s)=\dfrac{s^n}{p^{\binom{n+1}{2}}}\Lpq\{t^n\}\left(\dfrac{s}{p^n}\right).\]
Taking care that $\Dpq^n t^n=[n]_{p,q}!$, and $\Lpq\{1\}(s)=\dfrac{1}{s}$, it follows that 
\[\Lpq\{t^n\}\left(\dfrac{s}{p^n}\right)=p^{\binom{n+1}{2}}\dfrac{[n]_{p,q}!}{s^n}\Lpq\{1\}(1)=p^{\binom{n+1}{2}}\dfrac{[n]_{p,q}!}{s^{n+1}}.\]
Replacing $s$ by $sp^n$, we then have 
\[\Lpq\{t^n\}\left(s\right)=p^{\binom{n+1}{2}}\dfrac{[n]_{p,q}!}{s^{n+1}p^{n(n+1)}}=\dfrac{[n]_{p,q}!}{p^{\binom{n+1}{2}}s^{n+1}}.\]

\subsection{The $(p,q)$-Laplace transform of second kind}

\noindent Whereas in the previous sections we introduce the $(p,q)$-Laplace transform of the first kind and prove some of its important properties, in this section, we introduce the $(p,q)$-Laplace transform of the second kind. The main difference is at the level of the $(p,q)$-exponential used in the definition. The motivation of the next definition comes from the fact that when we transform the big $(p,q)$-exponential, the result remains in term of a series which we cannot simplify. \\

\noindent Let first introduce the $(p,q)$-Gamma function of the second king which will be useful.  

\begin{definition}
The $(p,q)$-Gamma function of the second kind is defined by
\begin{equation}
    \gamma_{p,q}(z)=q^{\frac{z(z-1)}{2}}\int_{0}^{\infty}t^{z-1}\epq(-pt)\dpq t, \;\; \Re(z)>0. 
\end{equation}
\end{definition}

\begin{proposition}
The $(p,q)$-Gamma function fulfils the following fundemental relation
\begin{equation}
\gamma_{p,q}(z+1)=[z]_{p,q}\gamma_{p,q}(z),
\end{equation}
moreover, for any non negative integer $n>0$, the following relation holds
\begin{equation}
\gamma_{p.q}(n+1)=[n]_{p,q}!.
\end{equation}
\end{proposition}

\begin{proof}
Let $z$ be a complex number such that $\Re(z)>0$, then we have 
\begin{eqnarray*}
   \gamma_{p,q}(z+1)&=& q^{\frac{z(z+1)}{2}}\int_{0}^{\infty}t^z\epq(-pt)\dpq t\\
   &=& -q^{\frac{z(z-1)}{2}}\int_{0}^{\infty}(qt)^z\Dpq\epq(-t)\dpq t\\
   &=&-q^{\frac{z(z-1)}{2}}\left\{ \left[t^z\epq(-t)\right]_0^\infty-[z]_{p,q}\int_{0}^{\infty}t^{z-1}\epq(-pt)\dpq t\right\}\\
   &=& [z]_{p,q}\gamma_{p,q}(z). 
\end{eqnarray*}
\end{proof}

\begin{definition}
For a given function $f(t)$, we define its $(p,q)$-Laplace transform of the first kind as the function 
\begin{equation}\label{pqlaplace2}
F(s)=\lpq\{f(t)\}(s)=\int_{0}^{\infty}f(t)e_{p,q}(-pts)\dpq t, \;\; s>0.
\end{equation}
\end{definition}

\begin{proposition}[Linearity]
By \eqref{pqlaplace2}, we have 
\begin{equation*}
\lpq\{\alpha f(t)+\beta g(t)\}=\alpha \lpq\{f(t)\}+\beta \lpq\{g(t)\}. 
\end{equation*}
\end{proposition}

\begin{proposition}
For any real number $\alpha>-1$, we have 
\begin{equation}
   \lpq\{t^\alpha\}(s)=\dfrac{\gamma_{p,q}(\alpha+1)}{q^{\frac{\alpha(\alpha-1)}{2}}s^{\alpha+1}}.
\end{equation}
\end{proposition}

\begin{proof}
By definition, one has 
\begin{eqnarray*}
   \lpq\{t^\alpha\}(s)&=& \int_{0}^{\infty} t^{\alpha}e_{p,q}(-pts)\dpq t\\
   &=& \dfrac{1}{s^{\alpha+1}}\int_{0}^{\infty}t^{\alpha}\epq(-pt)\dpq t\\
   &=& \dfrac{\gamma_{p,q}(\alpha+1)}{q^{\frac{\alpha(\alpha-1)}{2}}s^{\alpha+1}}.
\end{eqnarray*}
\end{proof}

\begin{proposition}
For $n\in\N$, we have have 
\begin{equation}\label{lpqtrans-power}
   \lpq\{t^n\}(s)=\dfrac{[n]_{p,q}!}{q^{\binom{n+1}{2}}s^{n+1}}.
\end{equation}
\end{proposition}

\begin{proof}
Clearly, we have for $n=0$
\[\lpq\{1\}(s)=\int_{0}^{\infty}e_{p,q}(-pts)\dpq t=-\dfrac{1}{s}\left[e_{p,q}(ts)\right]_0^{\infty}=\dfrac{1}{s}.\]
Next, for $n>0$, we have 
\begin{eqnarray*}
\lpq\{t^n\}(s)&=&\int_0^{\infty}t^ne_{p,q}(-pts)\dpq t\\
&=& -\dfrac{1}{q^ns}\int_{0}^{\infty}(qt)^n\Dpq e_{p,q}(-ts)\dpq t\\
&=& -\dfrac{1}{q^ns}\left\{  \left[ t^ne_{p,q}(-ts)\right]_{0}^{\infty}-[n]_{p,q}\int_{0}^{\infty} t^{n-1}e_{p,q}(-pts)\dpq t\right\}\\
&=& \dfrac{[n]_{p,q}}{q^n s}\lpq\{t^{n-1}\}(s).
\end{eqnarray*}
The proof then follows by induction. 
\end{proof}

\begin{proposition}
The following equation holds
\begin{eqnarray}
    \lpq \{E_{p,q}(at)\}(s) &=& \dfrac{q}{qs-a},\quad s>\left|\dfrac{a}{q}\right|. \label{TransformExp2}
\end{eqnarray}
\end{proposition}

\begin{proof}
We have 
\begin{eqnarray*}
\lpq\{E_{p,q}(at)\}(s)&=& \sum_{n=0}^{\infty}\dfrac{q^{\binom{n}{2}}a^n}{[n]_{p,q}!}\int_{0}^{\infty}t^ne_{p,q}(-pts)\dpq t\\
&=& \sum_{n=0}^{\infty}\dfrac{q^{\binom{n}{2}}a^n}{[n]_{p,q}!}\times \dfrac{[n]_{p,q}!}{q^{\binom{n+1}{2}}s^{n+1}}\\
&=& \dfrac{1}{s}\sum_{n=0}^{\infty}\left(\dfrac{a}{qs}\right)^n=\dfrac{q}{qs-a}.
\end{eqnarray*}
\end{proof}

\begin{corollary}\label{corCos}
The following equations hold
\begin{eqnarray*}
    \lpq\{Cos_{p,q}(at)\}(s)&=& \dfrac{q^2s}{(qs)^2+a^2},\quad s>\left|\dfrac{a}{q}\right|,\\
    \lpq\{Sin_{p,q}(at)\}(s)&=& \dfrac{qa}{(qs)^2+a^2},\quad s>\left|\dfrac{a}{q}\right|.
\end{eqnarray*}
\end{corollary}

\begin{proof}
The proof follows from the definitions \eqref{pqcosBig}, \eqref{pqsinBig} and equation \eqref{TransformExp2}. 
\end{proof}

\begin{corollary}
The following equations hold
\begin{eqnarray*}
    \lpq\{Cosh_{p,q}(at)\}(s)&=& \dfrac{q^2s}{(qs)^2-a^2},\quad s>\left|\dfrac{a}{q}\right|,\\
    \lpq\{Sinh_{p,q}(as)\}(s)&=& \dfrac{qa}{(qs)^2-a^2},\quad  s>\left|\dfrac{a}{q}\right|.
\end{eqnarray*}
\end{corollary}

\begin{proof}
The proof is similar to the proof of Corollary \ref{corCos}.
\end{proof}

\noindent Next, $f$ being a function, we provide some properties related to the $(p,q)$-derivative of the $(p,q)$-Laplace transform of $f$ and the $(p,q)$-Laplace transform of the $(p,q)$-derivative of $f$. 

\begin{theorem}[$(p,q)$-derivative of transforms]\label{theo-der-trans-2} For $n\in\N$, we have 
\begin{equation}\label{trans-tft2}
\mathcal{L}_{p,q}\{t^nf(t)\}(s)=(-1)^np^{\binom{n}{2}} \partialpq{s}{n}\left[F\left(p^{-n}s\right)\right]
\end{equation}
where $F(s)=\displaystyle{\lpq\{f(t)\}(s)}$. 
\end{theorem}

\begin{proof}
The result is obvious for $n=0$. Let $n\geq 1$, we have  
\begin{eqnarray*}
\partialpq{s}{n}[F(p^{-n}s)]&=& \int_{0}^{\infty}\! \partialpq{s}{n}\left[ \epq(-p^{-n+1}st)\right]f(t)\dpq t
\end{eqnarray*}
Using equation \eqref{nth-deriv-pqExp}, it follows that 
\begin{eqnarray*}
    \partialpq{s}{n}\left[ \epq(-p^{-n+1}st)\right]&=&\prod_{j=0}^{n-1}\left(-p^{n-1-j}t\right)\epq(-pst)\\
    &=&(-1)^np^{-\binom{n}{2}}t^n\epq(-pst).
\end{eqnarray*}
The proof is then completed.
\end{proof}

\noindent Note that \eqref{trans01} can be obtained using Theorem \ref{theo-der-trans-2}. Of course, taking $f(t)=1$ in \eqref{trans-tft2} and using \eqref{nth-deriv-invfunct}, we have $F(s)=\dfrac{1}{s}$ and
\[L_{p,q}\{t^n\}(s)=(-1)^np^{\binom{n}{2}}\ \partialpq{s}{n}\left[\dfrac{p^n}{s}\right]=(-1)^np^{\binom{n+1}{2}}\dfrac{(-1)^n[n]_{p,q}!}{(pq)^{\binom{n+1}{2}}s^{n+1}}=\dfrac{[n]_{p,q}!}{q^{\binom{n+1}{2}}s^{n+1}}.\]

\begin{corollary}
The following equation applies:
\begin{eqnarray*}
\lpq\{t^n\Epq(at)\}(s)&=&\dfrac{q^{n+1}p^{\binom{n+1}{2}}[n]_{p,q}!}{(q^{n+1}s-ap^n)(q^nps-ap^n)\cdots(q^2p^{n-1}s-ap^n)(pq^ns-ap^n)}\\
&=&\dfrac{q^{n+1}p^{\binom{n+1}{2}}[n]_{p,q}!}{\prod\limits_{k=0}^{n}\left(q^{n+1-k}p^ks-ap^n\right)}.
\end{eqnarray*}
\end{corollary}

\begin{proof}
The proof follows from \eqref{gen-nth-deriv-invfunct} and \eqref{trans-tft2}. 
\end{proof}

\begin{theorem}[Transform of the $(p,q)$-derivative]
For any nonnegative integer $n$, we have 
\begin{equation}\label{tr-pqder}
\lpq\left\{\Dpq^n f(t)\right\}=\dfrac{s^n}{q^{\binom{n+1}{2}}}\lpq\{f(t)\}\left(\frac{s}{q^n}\right)-\sum_{k=0}^{n-1}\dfrac{s^{n-1-k}}{q^{\binom{n-k}{2}}}\left(\Dpq^k f\right)(0).
\end{equation}
\end{theorem}

\begin{proof}
For $n=1$, we have
\begin{eqnarray*}
\lpq\{f(t)\}(s)&=&\int_{0}^{\infty}\Dpq f(t)\epq(-pst)\dpq t\\
&=&\left[ f(t)\epq(-st)\right]+s\int_{0}^{\infty}f(qt)\epq(-pst)\dpq t\\
&=& -f(0)+\dfrac{s}{q}\int_{0}^{\infty}f(t)\epq\left(-p\frac{s}{q}t\right)\dpq t\\
&=& -f(0)+\dfrac{s}{q}\lpq\{f(t)\}\left(\frac{s}{q}\right).
\end{eqnarray*} 
So the relation is true for $n=1$. Let $n\geq 1$, assume that \eqref{tr-pqder} holds true, then using the case $n=1$, we can write
\begin{eqnarray*}
\lpq\left\{\Dpq^{n+1} f(t)\right\}&=&-(\Dpq^nf)(0)+\dfrac{s}{q}\lpq\{\Dpq^n f(t)\}\left(\frac{s}{q}\right)\\
&=&-(\Dpq^nf)(0)+\dfrac{s}{q}\left\{\dfrac{s^n}{q^{\binom{n+1}{2}+n}}\lpq\{f(t)\}\left(\frac{s}{q^{n+1}}\right)\right.\\
&&\left.-\sum_{k=0}^{n-1}\dfrac{s^{n-1-k}}{q^{\binom{n-k}{2}+n-1-k}}\left(\Dpq^k f\right)(0)\right\}\\
&=& -(\Dpq^nf)(0)+\dfrac{s^{n+1}}{q^{\binom{n+1}{2}+n+1}}\lpq\{f(t)\}\left(\frac{s}{q^{n+1}}\right)\\
&&-\sum_{k=0}^{n-1}\dfrac{s^{n-k}}{q^{\binom{n-k}{2}+n-k}}\left(\Dpq^k f\right)(0)\\
&=&\dfrac{s^{n+1}}{q^{\binom{n+2}{2}}}\lpq\{f(t)\}\left(\frac{s}{q^{n+1}}\right)
-\sum_{k=0}^{n}\dfrac{s^{n-k}}{q^{\binom{n-k+1}{2}}}\left(\Dpq^k f\right)(0).
\end{eqnarray*}
The relation holds then true for each integer $n\geq 1$. 
\end{proof}
We now have another possibility to comput $\lpq\{t^n\}(s)$ using \eqref{tr-pqder}. Of course, applying \eqref{tr-pqder} to $f(t)=t^n$, we have 
\[\lpq\{\Dpq^n t^n\}(s)=\dfrac{s^n}{q^{\binom{n+1}{2}}}\lpq\{t^n\}\left(\frac{s}{q^n}\right).\]
Taking care that $\Dpq^n t^n=[n]_{p,q}!$, it follows that 
\[\lpq\{t^n\}\left(\frac{s}{q^n}\right)= q^{\binom{n+1}{2}}\dfrac{[n]_{p,q}!}{s^n}\lpq\{1\}(s)=\dfrac{[n]_{p,q}!q^{\binom{n+1}{2}}}{s^{n+1}}.\]
Replacing $s$ by $sq^n$, it follows that 
\[\lpq\{t^n\}(s)=\dfrac{[n]_{p,q}!q^{\binom{n+1}{2}}}{s^{n+1}q^{n(n+1)}}=\dfrac{[n]_{p,q}!}{q^{\binom{n+1}{2}}s^{n+1}}.\]

\section{Application of $(p,q)$-Laplace transform to certain $(p,q)$-difference equations}

\noindent As Laplace transform and $Z$-transform are largely applied in solving differential and difference equations respectively, and the $q$-Laplace transforms are applied to solve $q$-difference equations, the $(p,q)$-Laplace transforms are expected to play similar role but now in $(p,q)$-difference equations. The idea lying behind is always the same.  In this section, we show on few examples how the Laplace transforms introduced before can be used to solve some $(p,q)$-differential equations. 

Consider the problem of finding $f(t)$, where $f(t)$ satifies $(p,q)$-Cauchy problem
\begin{equation}\label{pqdiff1}
\Dpq f(t)+cf(pt)=0,\quad f(0)=1,
\end{equation}
where $c$ stands for a complex constant.

Applying the Laplace transform of the first kind to \eqref{pqdiff1}, we obtain 
\[ -f(0)+\dfrac{s}{p}\Lpq\{f(t)\}\left(\dfrac{s}{p}\right)+c\Lpq\{f(pt)\}(s)=0.\]
Next, usint equation \eqref{theo-dil-equation}, and the initial condition $f(0)=0$, we get 
\[ -1+\dfrac{s}{p}\Lpq\{f(t)\}\left(\dfrac{s}{p}\right)+\dfrac{c}{p}\Lpq\{f(t)\}\left(\dfrac{s}{p}\right)=0.\]
Hence, 
\[\Lpq\{f(t)\}\left(\dfrac{s}{p}\right)=\dfrac{p}{s+c},\]
and so
\[\Lpq\{f(t)\}\left(s\right)=\dfrac{p}{ps+c},\]
It follows that $f(t)=\epq(-ct)$. 

Now, consider the $(p,q)$-differential equation 
\begin{equation}\label{pqdiff2}
\Dpq h(t)-\lambda h(pt)=\epq(\lambda qt), \quad h(0)=0.
\end{equation} 
Applying the $(p,q)$-Laplace transform of first kind to \eqref{pqdiff2}, it follows that 
\[-h(0)+\dfrac{s}{p}\Lpq\{h(t)\}\left(\frac{s}{p}\right)-\dfrac{\lambda}{p}\Lpq\{h(t)\}\left(\frac{s}{p}\right)=\dfrac{p}{ps-\lambda q}.\] 
Simplififications give 
\[\Lpq\{h(t)\}\left(\frac{s}{p}\right)=\dfrac{p^2}{(s-\lambda)(ps-\lambda q)},\]
and finally, replacing $s$ by $ps$, we have 
\[\Lpq\{h(t)\}\left(s\right)=\dfrac{p^2}{(ps-\lambda)(p^2s-\lambda q)}.\]
So, clear $h(t)=t\epq(\lambda t)$. 

For the last example, we consider the classical $(p,q)$-oscillator 
\begin{equation}\label{pqoscillator}
\Dpq^2 f(t)+\omega^2 f(p^2t)=0, \quad \Dpq f(0)=A.,\;\; f(0)=B.
\end{equation}
Applying the $(p,q)$-Laplace transform of first kind to \eqref{pqoscillator}, it follows that 
\[ -A-\dfrac{Bs}{p}+\dfrac{s^2}{p^3}\Lpq\{f(t)\}\left(\frac{s}{p^2}\right)+\dfrac{\omega^2}{p^2}\Lpq\{f(t)\}\left(\frac{s}{p^2}\right)=0.\]
By an easy simplification, we get 
\[  \Lpq\{f(t)\}\left(\frac{s}{p^2}\right)=\dfrac{Bs+Ap}{p}\times \dfrac{p^3}{s^2+p\omega^2}.\]
It happens that 
\[\Lpq\{f(t)\}\left(s\right)=\dfrac{Bp^2s}{(ps)^2+\left(\frac{\omega}{\sqrt{p}}\right)^2} +A\dfrac{\sqrt{p}}{\omega}\dfrac{p\dfrac{\omega}{\sqrt{p}}}{(ps)^2+\left(\frac{\omega}{\sqrt{p}}\right)^2}.  \]
Hence, the solutions of the $(p,q)$-oscillators are 
\[f(t)=B\cos_{p,q}\left(\dfrac{\omega}{\sqrt{p}}t\right)+A\dfrac{\sqrt{p}}{\omega}\sin_{p,q} \left(\dfrac{\omega}{\sqrt{p}}t\right).\]

\section{Conclusion and perspectives}

\noindent In this work, we have introduced two Laplace transforms. Many properties of these new transforms have been proved. This works is certainly not complete and should be a starting point of many other works. For example, in future works, one could define the $(p,q)$-convolution product and compute its $(p,q)$-Laplace transform. This will of course anable to solve some $(p,q)$-convolution equations. Also, another work will be to find the inversion formula for these transform, so we could be able to solve many more $(p,q)$-differential equations.


%
%

\vspace{1cc}


{\small
\noindent P. Njionou Sadjang\\
University of Douala\\
Faculty of Industrial Engineering\\
E-mail: \texttt{pnjionou@yahoo.fr}

}\end{document}